\documentclass{amsart}

\usepackage[all]{xy}
\usepackage{longtable}
\usepackage{todonotes}

\usepackage{mathrsfs}
\usepackage{hyperref}
\usepackage{graphicx}
\hypersetup{
    colorlinks=true,       % false: boxed links; true: colored links
    linkcolor=blue,          % color of internal links
    citecolor=blue,        % color of links to bibliography
    filecolor=blue,      %t colo r of file links
    urlcolor=blue           % color of external links
}

\usepackage[backrefs,lite]{amsrefs}
\usepackage{booktabs}
\usepackage{tikz}

\newtheorem{theorem}{Theorem}[section]
\newtheorem{introtheorem}{Theorem}
\newtheorem{lemma}[theorem]{Lemma}
\newtheorem{proposition}[theorem]{Proposition}

\theoremstyle{definition}

\newtheorem{example}[theorem]{Example}

\theoremstyle{remark}
\newtheorem{remark}[theorem]{Remark}

\newcommand{\et}{{\rm \widetilde {E}}_8}
\newcommand{\ee}{{{\rm E}_8}}
\newcommand{\eec}{{\rm E}_8^\ast}
\newcommand{\E}{\operatorname{E}}
\newcommand{\D}{\operatorname{D}}
\newcommand{\A}{\operatorname{A}}
\newcommand{\zz}{\mathbb{Z}}
\newcommand{\pp}{\mathbb{P}}
\newcommand{\qq}{\mathbb{Q}}
\newcommand{\cc}{\mathbb{C}}

\newcommand{\Pic}{\operatorname{Pic}}
\newcommand{\Cl}{\operatorname{Cl}}
\newcommand{\cl}{\operatorname{cl}}

\newcommand{\Hom}{\operatorname{Hom}}
\newcommand{\Osh}{\mathcal{O}}

\newcommand{\kxp}{{K_X}^\perp}

\newcommand{\fib}{\pi}
\newcommand{\pro}{p}
\newcommand{\gru}{\alpha}
\newcommand{\gri}{\overline\gru}
\newcommand{\ik}{i}

\numberwithin{equation}{section}
\numberwithin{table}{section}
\numberwithin{figure}{section}

\begin{document}

\title{On minimal rational elliptic surfaces}

\thanks{
The first author was supported by Proyecto 
FONDECYT Regular, N. 1150732. 
}
\thanks{The second author was supported 
by EPSRC grant EP/K019279/1.}
\author{Antonio Laface}
\address{
Departamento de Matem\'atica, 
Universidad de Concepci\'on, 
Casilla 160-C,
Concepci\'on, Chile}
\email{antonio.laface@gmail.com}

\author{Damiano Testa}

\address{
Mathematics Institute, 
University of Warwick, 
Coventry, CV4 7AL, 
United Kingdom}
\email{adomani@gmail.com}

\date{\today}

\begin{abstract}
We construct $13$ projective $\qq$-factorial
Fano toric varieties and show that for any
minimal rational elliptic surface $X$ there is
one such toric variety $Z_X$ and a divisor 
class $\delta_X\in\Cl(Z_X)$ such that the number 
of $(-1)$-curves of $X$ equals the dimension of 
the Riemann-Roch space of $\delta_X$.
As an application we give the number of 
$(-1)$-curves of any such elliptic
fibration of Halphen index $2$.
\end{abstract}

\maketitle

\section*{Introduction}
Let $X$ be a smooth projective rational surface
which admits a morphism $\pi\colon X\to\pp^1$
whose general fiber is a smooth curve of genus
one. We say that $\pi$ is {\em minimal} if it does
not contract any $(-1)$-curve. 
Denote by $K_X$ the canonical divisor class on $X$;
by~\cite{BHPV}*{12.1} the map $\pi$ is
induced by a complete linear system of the
form $|{-mK_X}|$, with $m>0$. As a consequence,
the anticanonical divisor on $X$ is semiample with $(-K_X)^2=0$.
It is possible to show~\cite{DoCo} that
$\pi$ is the unique elliptic fibration on $X$,
that is $\pi$ admits a unique multiple fiber and that
this multiplicity is $m$. The number $m$
is the {\em Halphen index} of $X$.
It is well known that
the Mori cone of $X$ is finitely generated if and only
if the Cox ring is finitely generated and these two conditions
are equivalent to requiring $X$ to have a finite number
of $(-2)$-curves and $(-1)$-curves~\cite{AL,To}.
While the former
are exactly the prime components of the reducible 
fibers of $\pi$, thus easily enumerated, less
is known about the number of $(-1)$-curves of $X$.

The aim of this note is to provide a combinatorial 
approach to determine the number of such curves. 
More precisely, we  
associate to any such surface $X$ a $\qq$-factorial
projective toric variety $Z_X$ in the following way.
Let $F$ be the free 
abelian group with basis $\{f_C\}$ indexed by the 
prime components $C$ of the reducible fibers of $\pi$.
By the genus formula it follows that such components
are exactly the $(-2)$-curves of $X$.
The linear map
\[
 P_X
 \colon 
 F\to \dfrac{ \hphantom{{}^\bot} 
 {K_X}^\bot} {\langle K_X\rangle} = \ee\cong \zz^8
 \qquad {\textrm{determined by}}\qquad
 f_C\mapsto [C]
\]
defines the $1$-skeleton of the fan of a projective
toric variety $Z_X$. It is possible to show that there 
is a unique complete toric variety $Z_X$ with this 
skeleton (see Remark~\ref{fano}).
Recall that, given $E=\Hom(F,\zz)$ the dual of $F$,
the cokernel $Q_X$ of the dual map $P_X^*$ is the 
class map $Q_X\colon E\to \Cl(Z_X)$.
This allows us to define a homomorphism
\[
 \alpha\colon\Pic(X)\to\Cl(Z_X)
\]
by mapping any divisor class $[D]$ to the element
$(\Gamma\mapsto D\cdot\Gamma)$ of $E$
and then applying $Q_X$.
In what follows we will denote by $=_{\rm free}$ 
equality up to torsion.
Our main result is the following theorem.
\begin{introtheorem}
\label{main}
Let $\pi \colon X \to \mathbb{P}^1$ be a minimal
elliptic fibration on a rational surface with Halphen 
index $m$ and let $Z_X$ be the associated 
toric variety.
\begin{enumerate}
\item
\label{one}
All the $(-1)$-curves of $X$ are mapped to the 
same class $\delta_X\in\Cl(Z_X)$ by $\alpha$.
The set of $(-1)$-curves of $X$ is in bijection 
with the set of integral points of the Riemann-Roch 
polyhedron
\[
 \Delta(\delta_X)
 :=
 Q_X^{-1}(\delta_X)\cap (E_\qq)_{\geq 0}.
\]
Moreover the coordinates of an integral
point of the polyhedron $\Delta(\delta_X)$
are the intersection multiplicities of the
corresponding $(-1)$-curve with 
the prime components of the fibers of $\pi$.
\item
\label{two}
For any reducible fiber $F_i$ of $\pi$ there exists 
an element $w_i\in\Cl(Z_X)$ such that
for any prime component $C$ of $F_i$ 
\[
 Q_X(C)=_{\rm free} n_C\, w_i
\]
where $n_C$ is the positive integer which
appear as a coefficient of the root $C$ in the unique
integer relation between the roots of the affine Dynkin 
diagram corresponding to $F_i$.
\item
\label{three}
Let $F_1,\dots,F_r$ be the reducible fibers of $\pi$,
let $w_1,\dots,w_r\in\Cl(Z_X)$ be as in (2) and let
$F$ be the multiple fiber of $\pi$.
Then the following holds
\[
 \delta_X
 =_{\rm free}
 \begin{cases}
  m(w_1+\dots+w_r) & \text{ if $F$ is not in } \{ F_1 , \ldots , F_r\}, \\
  m(w_1+\dots +w_r) - (m-1)w_i & \text{ if $F=F_i$}.
 \end{cases}
\]
\end{enumerate}
\end{introtheorem}

The paper is organised as follows. In Section~\ref{sec:1}
we prove Theorem~\ref{main}. Section~\ref{sec:2}
is devoted to rational elliptic surfaces of Halphen 
index two and the enumeration of their $(-1)$-curves
as an application of Theorem~\ref{main}.
Finally in the Appendix we list the degree matrices
of the thirteen $\qq$-factorial Fano toric varieties
related to root lattices and which appear in the 
theory exposed in this note.

\section{$(-1)$-curves and lattice points}
\label{sec:1}

Let $\pi\colon X\to\pp^1$ be a minimal elliptic fibration
on a smooth projective rational surface.
Denote by $\ee$ the lattice $\kxp / \langle K_X\rangle$ and by 
$\pro \colon \kxp \to \ee$ the canonical quotient map; as the 
notation suggests, the lattice $\ee$ is isometric to the negative definite 
$\ee$-lattice.  
We let $F$ be the free abelian group with basis indexed 
by the $(-2)$-curves on $X$ and we let $\cl_X \colon F \to \kxp$ 
be the homomorphism assigning to each basis vector 
the corresponding $(-2)$-class; we denote by $P_X$ the 
composition $\pro \circ \cl_X$.  The images of the rays 
of the standard basis of $F$ under the homomorphism 
$P_X$ define a one-dimensional fan in $\ee$. 
We will show in Remark~\ref{fano}
that there exists a unique projective $\qq$-factorial
toric variety $Z_X$ whose one-dimensional
cones are the cones of the fan defined by $P_X$.
We summarise the situation in the following 
commutative diagram with exact rows
\begin{equation} \label{diag:fan}
\begin{array}{c}
\xymatrix{
 F\ar@{=}[d]\ar[r]^-{\cl_X} & \kxp \ar[d]^-\pro\ar[r]^-\tau & A\ar[r]\ar[d] & 0 \\
 F \ar[r]^{P_X} & \ee\ar[r]^-{\bar\tau} & A/\langle h\rangle\ar[r] & 0 
}
\end{array}
\end{equation}
where $A$ is the cokernel of $\cl_X$ and
$h:=\tau(K_X)$.
\begin{remark}
The image of the homomorphism $P_X$ is
a root sublattice $\Lambda$ of $\ee$. 
According to~\cite{OgSh}*{Thm. 3.2 and Thm. 3.3}
there is  a unique isometric embedding of $\Lambda$
into $\ee$ modulo the action of the Weyl group
of $\ee$ unless $\Lambda$ is one of the following
lattices: $A_7$, $2A_3$, $A_5+A_1$, $A_3+2A_1$, $4A_1$.
In particular, if the rank of $\Lambda$ is eight, then
the embedding is unique.
\end{remark}

Let $E=\Hom(F,\zz)$ be the dual of $F$ 
and let $Q_X \colon E \to \Cl(Z_X)$ be the 
cokernel of the dual map $P_X^* \colon \eec \to E$.
Dualising the exact sequences in~\eqref{diag:fan} 
we obtain the following commutative diagram with 
exact rows
\begin{equation} \label{diag:cl}
\begin{array}{c}
\xymatrix{
 (\kxp)^*\ar[r]^-{\cl_X^*} & E \ar[r]\ar@{=}[d] 
 & \dfrac{\Cl(Z_X)}{\langle\delta_X\rangle}\ar[r] & 0\\
(\ee)^* \ar[r]^-{P_X^*}\ar[u]^-{\pro^*} & E\ar[r]^-{Q_X}
& \Cl(Z_X)\ar[r]\ar[u]^-\tau & 0 
}
\end{array}
\end{equation}
where $\delta_X\in\Cl(Z_X)$ generates the kernel 
of $\tau$ and is defined by
\[
 \delta_X
 =
 (Q_X\circ\cl_X^*)(m),
\] 
where $m\in (\kxp)^*$ generates the cokernel
of $p^*$, that is $m$ is any projection of the 
homomorphism $i\colon\zz\to\kxp$ defined by 
$1\mapsto -K_X$.
Observe that $\delta_X$ does not depend
on the choice of $m$.
%Also, if $D$ is a curve of $X$ such that $-K_X\cdot D = 1$, then the assignment $1 \mapsto D$ defines 

Given a class $w\in\Cl(Z_X)$ we recall that the
{\em Riemann-Roch polytope} of $w$ is
the convex polytope
\[
 \Delta(w) := Q_Z^{-1}(w)\cap (E_\qq)_{\geq 0},
\]
where $(E_\qq)_{\geq 0}$
is the positive orthant of the rational vector
space $E_\qq = E\otimes_\zz\qq$.
Assume now that $D$ is a $(-1)$-curve of
$X$ and denote by $e_D$ the linear function
$F\to\zz$ defined by $f\mapsto \cl_X(f)\cdot D$,
where the product is induced by the intersection
pairing of $X$.
We therefore obtain a function 
\[
 \varepsilon \colon \bigl\{ (-1){\textrm{-curves of }} X \bigr\} \longrightarrow E
 \qquad
 D\longmapsto e_D.
\]

The following lemma is needed in the proof of 
Lemma~\ref{mlemma} (see also~\cite[Proposition~3.3]{LH} 
for another proof).

\begin{lemma}
\label{lem:m1}
Let $\fib \colon X \to \mathbb{P}^1$ be a minimal
elliptic fibration on a  rational surface and let $D$ 
be a divisor on $X$ such that $D^2=D\cdot K_X=-1$
and $D\cdot C\geq 0$ for any $(-2)$-curve of $X$. 
Then $D$ is linearly equivalent to a $(-1)$-curve.
\end{lemma}
\begin{proof}
Since the elliptic fibration $\pi$ is minimal, the 
anticanonical divisor $-K_X$ is linearly equivalent to a 
multiple of a fiber and thus it is nef.  
By the Riemann-Roch theorem either $D$ or ${K_X-D}$ 
is linearly equivalent to an effective divisor.
The equality $-K_X \cdot (K_X-D) = -1$ and the fact that 
the divisor $-K_X$ is nef allow us to conclude that 
the first possibility occurs. Thus, without loss 
of generality, we can assume $D$ to be effective.

By the genus formula and the fact that $-K_X$ is nef, 
the only  curves of $X$ which are orthogonal
to $-K_X$ are $(-2)$-curves, while the only ones
which have intersection $1$ with $-K_X$ are the
$(-1)$-curves. Thus the equality $-K_X\cdot D=1$ 
implies $D = \Gamma + D_1$, where $\Gamma$
is a $(-1)$-curve and $D_1$ is a non-negative sum 
of $(-2)$-curves. From $(\Gamma+D_1)^2=D^2=-1$
we deduce the following
\[
 0\leq (D_1+\Gamma)\cdot D_1= -\Gamma\cdot D_1\leq 0,
\]
where the first inequality is due to the hypothesis
on $D=(D_1+\Gamma)$ and the definition of $D_1$, 
while the second inequality is due to the fact that
$\Gamma$ is not a component of $D_1$.
Thus $(D_1)^2=\Gamma\cdot D_1=0$ and,
since $-K_X\cdot D_1=0$, by the Hodge index theorem
we conclude that $D_1$ is linearly equivalent to 
$-\alpha K_X$ for some rational number $\alpha$. 
Using the identities $-1=D^2=(\Gamma-\alpha K_X)^2=-1+2\alpha$
we deduce that $\alpha=0$ and thus $D=\Gamma$ is
a $(-1)$-curve.
\end{proof}

\begin{lemma}\label{mlemma}
The function $\varepsilon$ induces a bijection 
between the set of $(-1)$-curves on $X$ and 
the set of integral points of the polytope
$\Delta(\delta_X)$.
\end{lemma}

\begin{proof}
We begin by showing that the image of $\varepsilon$ is 
contained in $\Delta(\delta_X)$. Indeed, if $D$ is a
$(-1)$-curve of $X$ the homomorphism 
$e_D\colon F\to\zz$ factors through  $\kxp$
by definition and it defines a projection
$m_D\colon \kxp\to\zz$ of $i\colon\zz\to\kxp$
since $-K_X\cdot D = 1$.
Thus $Q_X(e_D)=\delta_X$ by our definition
of $\delta_X$. Moreover,
for each element $f$ of the standard basis 
of $F$, we have
\[
 e_D(f) = \cl_X(f) \cdot C \geq 0,
\]
where the last inequality is due to the fact that 
$\cl_X(f)$ and $D$ correspond to distinct 
irreducible curves on $X$.  We deduce that the 
image of $\varepsilon$ is contained in 
$\Delta(\delta_X)$.

Now, we show that $\varepsilon$ is injective.  
Let $D,D'$ be $(-1)$-curves on $X$ and suppose 
that $e_D=e_{D'}$ holds.  
It follows that the difference $D-D'$ is proportional 
to $K_X$.  We deduce that 
$0 = (D-D')^2 = -2 - 2 D \cdot D'$, so that 
$D$ and $D'$ coincide.

Finally, let $e$ be an integral point of $\Delta(\delta_X)$.  
Therefore, we write $e = e_D + P_X^*(\gamma)$, 
where $D$ is a $(-1)$-curve on $X$ 
and $\gamma$ is an element of $\eec$.  
We deduce that $e$ factors through $\cl_X$ 
inducing a retraction $m_e \colon \kxp \to \mathbb{Z}$ 
of $i$.  Using the perfect pairing 
\[
 \xymatrix{
  \kxp \times \dfrac{\Pic(X)}{\langle-K_X\rangle}
  \ar[r] &
  \mathbb{Z},
 }
\]
we see that $m_e$ is multiplication by an element
$[D]+\mathbb{Z}(-K_X)$, where $D \cdot (-K_X)=1$
and for each $(-2)$-curve $\Gamma$ on $X$ the 
inequality $D \cdot \Gamma \geq 0$ holds.
Let $r$ be an integer such that $(D+r K_X)^2 = D^2 - 2r = -1$; 
such an integer $r$ exists since $D^2$ is odd by the
adjunction formula.  By Lemma~\ref{lem:m1}, 
the class of $D+rK_X$ is the class of a 
$(-1)$-curve and we are done.
\end{proof}

We now make explicit the homomorphism 
$Q \colon E \to \Cl(Z)$.
We begin with a sublattice $\Lambda$ of $\ee$ 
spanned by roots of $\ee$.
Let $\Lambda = \Lambda_1 \oplus \cdots \oplus \Lambda_r$ 
be a decomposition of $\Lambda$ as a direct sum of irreducible 
root lattices.
We fix a subset $\{\lambda_1 , \ldots , \lambda_s\}$
of the lattice $\Lambda$ consisting of simple roots
chosen as follows.
First of all we choose a basis of $\Lambda_1$ 
indecomposable roots together with the root corresponding 
to the affine vertex in the extended Dynkin diagram of 
$\Lambda_1$, equivalently, the last root is the opposite 
of the longest root of $\Lambda_1$. Then we proceed
similarly with $\Lambda_2$ and so on.
Let $F$ be the free abelian group of rank $s$ 
with basis $f_1 , \ldots , f_{s}$.  
Let $P \colon F \to \ee$ be the homomorphism 
defined by $f_1 \mapsto \lambda_1 , \ldots , f_{s} \mapsto \lambda_{s}$ 
and let $E=\Hom(F,\zz)$ be the dual lattice of $F$. 
The homomorphism $P$ defines a toric variety 
$Z$ and an exact sequence 
\[
 \xymatrix@1{
   \eec\ar[r]^-{P*} & E\ar[r]^-Q & \Cl(Z)\ar[r] & 0.
  } 
\]
as in~\eqref{diag:cl}. 
Let $\fib \colon X \to \mathbb{P}^1$ be a  
minimal rational elliptic surface with singular fibers 
of type $\widetilde{\Lambda}_1 , \ldots , \widetilde{\Lambda}_r$.  
We denote the $(-2)$-curves of $X$ by 
$C_1 , \ldots , C_{s}$ where the numbering is 
compatible with the numbering of the roots 
$\lambda_1 , \ldots , \lambda_{s}$ given above.
It follows from the definitions that the homomorphisms 
$P$ and $P_X$ coincide, and thus that also 
$Q$ and $Q_X$ coincide.
With the above notation the matrix $Q$
has the following form
\begin{equation}
\label{Q-mat}
 Q
 =
 \left[
  \begin{array}{cccc}
   Q_1 & 0 & \cdots & 0\\
   0 & Q_2 & \cdots & 0\\
   \vdots & \ddots & \ddots & \vdots\\
   0 & \cdots & 0 & Q_r\\
   \hline
   \multicolumn{4}{c}{\text{torsion}}
  \end{array}
 \right]
\end{equation}
where each $Q_i$ is the row vector of length
${\rm rk} \, \Lambda_i+1$ defined
by the unique relation between the simple roots 
of the chosen basis of $\Lambda_i$
together with the root corresponding to its
affine vertex.
In these coordinates the matrix $Q$ defines
a map $E\to\zz^r\oplus (\ee/\Lambda)$.

\begin{proof}[Proof of Theorem~\ref{main}]
The first part of~\eqref{one} follows from Lemma~\ref{mlemma}.
For the second part observe that a $(-1)$-curve $D$
gives an element $e_D\in E$ by taking intersection
with $D$. Thus the statement follows from
the fact that the canonical basis of $F$
consists of the prime components of the 
fibers of $\pi$.

To prove items~\eqref{two} and~\eqref{three}
we begin by recalling that the class $\delta_X\in\Cl(Z_X)$ 
is defined by any
projection $m\colon\kxp\to\zz$ of the inclusion 
map $i\colon \zz\to\kxp$ given by $1\mapsto -K_X$.
In particular such a projection is obtained by
intersecting with a $(-1)$-curve $D$ of $X$,
that is $m(C)=C\cdot D$ for any $C\in\kxp$.
Due to the above presentation of the matrix $Q$
the integers in the row corresponding to the vector $Q_i$ are
exactly the multiplicities of the unique primitive
vector of self-intersection zero in the root
lattice $\widetilde\Lambda_i\subset\kxp$
whose class is thus the class of $-mK_X$
whenever the corresponding $(-2)$-curves
form a fiber of $\fib$ or it is $-K_X$ whenever
these curves form the support of the non-reduced
fiber. Thus items~\eqref{two} and~\eqref{three} follow.
\end{proof}

\begin{remark}\label{fano}
Due to~\eqref{Q-mat} the free part of the
grading of any variable of the Cox ring
of the toric variety $Z$ is one of the canonical
vectors $w_1,\dots,w_r$ of $\zz^r$ and any 
such vector occurs as a degree of at least 
two variables. As a consequence, by
~\cite[Proposition~2.4.2.6]{ADHL}, there is 
a unique projective variety $Z$ with the
given Cox ring and such that
\[
 {\rm SAmple}(Z) = {\rm Mov}(Z) = {\rm Eff}(Z)
 = {\rm cone}(w_1,\dots,w_r).
\]
It follows that any class in the relative interior 
of the above cone defines the same $\qq$-factorial
projective toric variety $Z$. This is the case for
the anticanonical class $-K_Z$ so that
the variety $Z$ is Fano.
\end{remark}

\begin{remark}
Recall~\cite{MiSt}*{Lemma~8.16} that if $Z$ 
is a toric variety with Cox ring $\mathcal{R}(Z)$
and grading matrix $Q:=[w_1,\dots,w_n]$, then 
the multigraded Hilbert series of
$\mathcal{R}(Z)$ is
\[
 H(\mathcal{R}(Z),t) 
 =
 \frac{1}{(1-t^{w_1})\cdots(1-t^{w_n})},
\]
where $t^w$ is the character of the quasi-torus
${\rm Spec}(\cc[\Cl(Z)])$ corresponding to $w\in\Cl(Z)$.
Hence Table~\ref{gradings} and 
Theorem~\ref{main}
allow us to determine the number of $(-1)$-curves
on any such surface.
\end{remark}

\begin{example}\label{ex:D8}
Let $X$ be a  minimal rational elliptic surface of type
$D_8$ of Halphen index $2$ and whose multiple
fiber is irreducible. According to the table appearing in
the Appendix, the Hilbert series of the Cox ring 
of the corresponding toric variety $Z$ is
\[
 \frac{1}{(1-t)^2(1-tu)^2(1-t^2)^3(1-t^2u)^2}
\]
where the two variables are in $\cc[t,u]/\langle u^2-1\rangle$.
Expanding the above series, a direct calculation
shows that the coefficient of the term $t^2$ is $9$,
while that of $t^{2}u$ is $6$.
Both cases are realized as shown in the next
paragraphs.

{\em Case 1}.
Choose nine points $p_1,\dots,p_9$ on a smooth
plane cubic curve $C$ with flex at $p_1$
such that the classes of $p_2-p_1,\dots,p_9-p_1$
are the following elements of $\zz/2\zz\oplus\zz/2\zz$: 
$00,00,01,01,01,10,10,11,11$. For example we can take
$C:=V(x^3 - y^2z - xz^2)$,
$p_1=( 0, 1, 0 )$, $p_3=(-1, 0, 1)$, $p_6
= (0,0,1)$, $p_8=(1,0,1)$ and the remaining points
to be infinitely near to the given ones according
to the classes $p_i-p_1$ written above.
The corresponding pencil of sextics curves 
is generated by $(x^3 - y^2z - xz^2)^2$ and
$x^2 y^2 (x - z)(x + z)$.
The classes of the $(-1)$-curves of $X$ are 
the rows of the following matrix.

{\footnotesize
\[
 \left[
 \arraycolsep=3pt
 \begin{array}{rrrrrrrrrr}
  0&0&1&0&0&0&0&0&0&0\\
  0&0&0&0&0&1&0&0&0&0\\
  0&0&0&0&0&0&0&1&0&0\\
  0&0&0&0&0&0&0&0&0&1\\
  1&-1&-1&0&0&0&0&0&0&0\\
  2&0&0&-1&-1&-1&-1&-1&0&0\\
  2&0&0&-1&-1&-1&0&0&-1&-1\\
  4&-2&-2&-1&-1&-1&-2&0&-1&-1\\
  4&-2&-2&-1&-1&-1&-1&-1&-2&0
 \end{array}
 \right]
\]
}

The first column is the degree of the plane model, 
while the remaining columns are the negative of the 
multiplicities at the nine points. For example the first 
row is an exceptional divisor while the last row is a 
plane quartic with three double points and through 
other five simple points. \\

{\em Case 2}.
Choose nine points $p_1,\dots,p_9$ on a smooth
plane cubic curve $C$ with flex at $p_1$
such that the classes of $p_2-p_1,\dots,p_9-p_1$
are the following elements of $\zz/4\zz$: 
$0,0,2,2,2,1,1,3,3$. For example we can take
$C:=V(y^2z + xyz + yz^2 - x^3 - x^2z)$,
$p_1=( 0, 1, 0 )$, $p_3=(-1, 0, 1)$, $p_6
= (0,0,1)$, $p_8=(0,-1,1)$ and the remaining points
to be infinitely near to the given ones according
to the classes $p_i-p_1$ written above.
The corresponding pencil of sextics curves 
is generated by $(y^2z + xyz + yz^2 - x^3 - x^2z)^2$
and $x^2(x+z)^2y(x+y+z)$. The classes of the
$(-1)$-curves of $X$ are the rows of the following matrix.

{\footnotesize
\[
 \left[
 \arraycolsep=3pt
 \begin{array}{rrrrrrrrrr}
  0&0&1&0&0&0&0&\hphantom{-}0&0&\hphantom{-}0\\
  0&0&0&0&0&1&0&0&0&0\\
  0&0&0&0&0&0&0&1&0&0\\
  0&0&0&0&0&0&0&0&0&1\\
  1&-1&-1&0&0&0&0&0&0&0\\
  2&0&0&-1&-1&-1&-1&0&-1&0\\
 \end{array}
 \right]
\]
}

\end{example}

\section{Rational elliptic surfaces}
\label{sec:2}

Let $X$ be a smooth projective rational
surface with $-K_X$ nef and $(K_X)^2=0$.
Given a curve $C$ in the linear system $|{-K_X}|$
we can construct the homomorphism of groups
\[
 \gru \colon \kxp \to {\rm Pic}^0(C)
 \qquad
 w\mapsto\imath^* w,
\]
where $\imath\colon C\to X$ is the inclusion.
We denote by $G$ the image of $\gru$
and by $\ik\colon\zz\to\kxp$ the
homomorphism defined by $1\mapsto -K_X$
and we let $H$ be the subgroup of $G$
generated by the images via $\gru$ of the 
classes of all the irreducible components of $C$.
Let $\kappa\colon\zz\to H$ be the homomorphism
defined by $1\mapsto h=\gru(-K_X)$.
Using the isomorphism between the lattices
$\ee$ and $\kxp / \zz(-K_X)$ we define
$\pro\colon\kxp\to\ee$ to be the corresponding projection.
We summarize these definitions in the 
commutative diagram
\begin{equation}\label{main-diag}
 \begin{array}{c}
 \xymatrix{
  0\ar[r] & \zz \ar[r]^-\ik \ar[d]^-\kappa
  & \kxp \ar[r]^\pro \ar[d]^\gru 
  & \ee %\ar@/^1pc/[l]^{\sec} 
  \ar[r] \ar[d]^\gri & 0 \\
  0\ar[r] & \langle h\rangle \ar[r]\ar@{^(->}[d] & G \ar[r]\ar@{=}[d] 
  & G/\langle h\rangle \ar[r]\ar@{->>}[d]^-\beta & 0 \\
  0\ar[r] & H \ar[r] & G \ar[r] & G/H\ar[r] & 0 \\
}
\end{array}
\end{equation}
with exact rows.

We recall the proofs of the following well-known 
lemmas for the sake of completeness.
See also~\cite{DoCo} and~\cite{Har}.

\begin{lemma} \label{res}
The order of the element $h$ is finite if and only if the 
surface $X$ is a minimal rational elliptic surface.
Moreover, if $h$ has finite order $m$, 
then the linear system $|{-mK_X}|$ is the
elliptic pencil.
\end{lemma}
\begin{proof}
Given $C\in |{-K_X}|$ there is an exact sequence
of sheaves
\begin{equation}
\label{seqK}
 \xymatrix{
  0\ar[r] & \Osh_X(K_X)\ar[r] & \Osh_X\ar[r] & \Osh_C\ar[r] & 0.
 }
\end{equation}
Let $n>0$ be an integer.
By taking tensor product with $\Osh_X(-nK_X)$ and
passing to the long exact sequence in cohomology, 
we see that the inequality $h^0(X,-nK_X)>1$ holds if and only 
$\gru(-nK_X)$ is trivial in ${\rm Pic}^0(C)$, that is
if and only if the class $h=\alpha(-K_X)$ is torsion. This is
equivalent to requiring that $h$ has finite order.
\end{proof}

\begin{lemma} \label{effective}
Let $R$ be a root of $\ee$. Then one 
of the following possibilities occurs:
\begin{enumerate}
\item
\label{case-1}
if $\overline\alpha(R)=0$ then $p^{-1}(R)$
contains an effective class;
\item
\label{case-2}
otherwise the only effective
classes in $p^{-1}(R)$ are components of $C$.
\end{enumerate}
\end{lemma}
\begin{proof}
Assume that we are in case~\eqref{case-1}.
Let $D'$ be a class in $\pro^{-1}(R)$, so that
$D'^2=-2$ and $D'\cdot K_X=0$,
and let $m$ be an integer such that $\gru(D')=mh$.
We denote by $D$ the class $D'-mK_X$,
so that $\gru(D)$ is trivial. By abuse
of notation, denote by $D$ also a divisor of $X$
whose class is $D$ and form the exact sequence 
of sheaves
\[
 \xymatrix{
  0\ar[r] & \Osh_X(D+K_X)\ar[r] & \Osh_X(D)\ar[r] & \Osh_C(D)\ar[r] & 0.
 }
\]
Since $\gru(D)$ is trivial, 
the first two cohomology groups of $\Osh_C(D)$
have dimension one.
The above exact sequence and 
the vanishing of the Euler characteristic of 
$\Osh_X(D)$ immediately imply that at least
one group $H^0(X,D)$, $H^2(X,D+K_X)^\vee\cong H^0(X,-D)$
is non-trivial. Thus either $D$ or $-D$ is linearly 
equivalent to an effective divisor $\bar D$.
If $D$ is linearly equivalent to $\bar D$,
then we are done. Otherwise there exists
a positive integer $n$ such that $D-nK_X
\sim -\bar D-nK_X$
is linearly equivalent to an effective divisor
and again the statement follows.

To prove~\eqref{case-2} it suffices to observe
that any effective class of $p^{-1}(R)$
is the class of a non-negative sum of 
$(-2)$-curves and that
if $\Gamma$ is a $(-2)$-curve of $X$ not 
contained in $C$, then $\Gamma$ is disjoint 
from $C$ and hence $\gru([\Gamma])=0$.
\end{proof}

Given a smooth rational surface
$X$ with $-K_X$ nef and ${K_X}^2=0$, denote by
$\Lambda_X$
the sublattice of $\kxp$ spanned
by the classes of $(-2)$-curves
of $X$:
\[
 \Lambda_X
 =
 \langle
  R \,:\, R \text{ is the class of a }(-2)\text{-curve}
 \rangle.
\]

\begin{remark}
By the adjunction formula $\Lambda_X$ is a 
negative semidefinite sublattice of $\kxp$
and the quotient $\Lambda_X/\langle K_X\rangle$
is a negative definite lattice.
By Lemma~\ref{effective} each fiber of the 
quotient homomorphism $\Lambda_X\to\Lambda_X/\langle K_X\rangle$
contains at most a finite number of 
classes of $(-2)$-curves. Since the 
image of the set of classes of $(-2)$-curves
of $X$ is a subset of the finite set of roots
of $\Lambda_X/\langle K_X\rangle\subseteq\ee$
we conclude that
$X$ contains finitely many $(-2)$-curves.
\end{remark}

We assume now that $C$ is a smooth plane cubic, 
$p_1,\ldots,p_9$ are nine points on $C$,
not necessarily distinct, and $p_1$ is a flex.
Denote by $X$ the blow up of $\mathbb{P}^2$ at the nine points 
$p_1 , \ldots , p_9$, by $E_1,\ldots,E_9$ the corresponding 
sequence of $(-1)$-classes in $\Pic (X)$ and by $L$
the class of the pull-back of a line.
Recall from~\eqref{main-diag} that 
the homomorphism $\alpha$ is given by
\begin{equation}
 \label{restr}
 \gru \colon \kxp \to G
 \qquad
 \qquad
 d L - (m_1 E_1 + \cdots + m_9 E_9)  \mapsto \sum _{i = 2}^9 m_i p_i
\end{equation}
since $L|_C\sim 3p_1$ and $p_1$ is the origin
of the elliptic curve $(C,p_1)$.
The {\em characteristic sequence} of 
the pair given by the points $p_1,\dots,p_9$
and the smooth plane cubic curve $C$ is the 
sequence $[p_1,\dots,p_9]$ in $\Pic^9(C)$.

\begin{proposition}
Let $\Lambda$ be a sublattice of $\ee$
of type ADE and of finite index.
Then there exists a minimal rational elliptic
surface $X$ such that $\Lambda_X/\langle K_X\rangle$ is 
isomorphic to $\Lambda$.
An example of such a surface for
any such $\Lambda$ is given in Table~\eqref{tab}.

{\footnotesize
\begin{longtable}{l|l|l}
\label{tab}
Lattice & Group & Characteristic sequence\\
\midrule
$\E_8$ & $\langle 0\rangle$ & $[0,0,0,0,0,0,0,0,0]$ \\
$\D_8$ & $\zz/2$ & $[0,0,0,0,0,1,1,1,1]$ \\
$\E_7+\A_1$ & $\zz/2$ & $[0,0,0,0,0,0,0,1,1]$ \\
$\A_8$ & $\zz/3$ & $[0,0,0,0,1,1,1,1,2]$ \\
$\E_6+\A_2$ & $\zz/3$ & $[0,0,0,0,0,0,1,1,1]$ \\
$\A_7+\A_1$ & $\zz/4$ & $[0,0,0,0,1,1,2,2,2]$ \\
$\D_5+\A_3$ & $\zz/4$ & $[0,0,0,0,0,1,1,1,1]$ \\
$2\A_4$ & $\zz/5$ & $[0,0,0,0,1,1,1,1,1]$ \\
$\A_5+\A_2+\A_1$ & $\zz/6$& $[0,0,0,1,1,1,1,1,1]$ \\
$\D_6+2\A_1$ & $\zz/2\oplus\zz/2$ & $[0,0,0,0,0,10,10,01,01]$ \\
$2\D_4$ & $\zz/2\oplus\zz/2$ & $[0,0,0,10,10,01,01,11,11]$ \\
$2\A_3+2\A_1$ & $\zz/2\oplus\zz/4$ & $[0,0,0,10,10,01,01,01,01]$\\
$4\A_2$ & $\zz/3\oplus\zz/3$ & $[0,0,0,10,10,10,01,01,01]$\\[5pt]
\caption{Characteristic sequences for maximal rank lattices}\\
\end{longtable}
}

\end{proposition}
\begin{proof}
By hypothesis the group $G:=\ee/\Lambda$
is finite. Moreover, by the classification
of sublattices of $\ee$ generated by roots, the 
lattice $\Lambda$ is isomorphic to one of the 
lattices of first column of the above table.
Assume now that there exists a characteristic 
sequence $[p_1,\dots,p_9]$ of points on
a smooth plane cubic curve $C$ such that
$\Lambda$ is contained in the kernel of
the homomorphism $\alpha$ of~\eqref{restr}
and moreover $\alpha$ is surjective.
By Lemma~\ref{res} and the fact that
$G$ is a finite group we deduce that the
surface $X$ is elliptic and minimal. 
Moreover the prime components of the unique
elliptic fibration $\pi\colon X\to\pp^1$ form
a Coxeter-Dynkin diagram of type $\tilde\Lambda$
by Lemma~\ref{effective}, which gives
the statement.

Hence it suffices to find a characteristic
sequence realizing each of the thirteen
lattices; this is shown in the third
column of the table.
\end{proof}

As an application of our results we determine the 
number of $(-1)$-curves on any minimal elliptic
surface of Halphen index $2$.
In what follows, given a root lattice $\Lambda$
we denote by $\tilde\Lambda$ the corresponding 
affine root lattice and by $v_\Lambda$ a primitive 
vector of null square of $\tilde\Lambda$.

\begin{lemma}
\label{lifting}
Let $\imath\colon \Lambda\to\ee$ be an embedding of a
root lattice of rank eight into $\ee$. Then
the embeddings $\tilde\imath\colon\tilde\Lambda\to\et$, 
which are compatible with $j$, are classified by the pairs $(m,\xi)$, 
where $m$ is a positive integer and 
$\xi\in {\rm Ext}^1_{\zz}(\ee/\Lambda,\zz/m\zz)$.
\end{lemma}
\begin{proof}
Observe that there are two isometries
$\tilde\Lambda\cong\Lambda\oplus
\langle v_\Lambda\rangle$ and $\et\cong\ee\oplus
\langle v_\ee\rangle$. Any isometric embedding
$\tilde\imath\colon\tilde\Lambda\to\et$ must map $v_\Lambda$
to an integer multiple $mv_\ee$ of $v_\ee$ being 
these the only vectors of null square of both 
lattices. Moreover, up to sign change, we can 
assume $m$ to be a positive integer. Then the
statement is consequence of the following
commutative diagram with exact rows and columns.
\[
 \xymatrix@1{
   & 0\ar[d] & 0\ar[d] & 0\ar[d] \\
  0\ar[r] & \zz\ar[d]\ar[r]^-{\cdot m}& \zz\ar[r]\ar[d] 
  & \zz/m\zz\ar[d]\ar[r] & 0\\
  0\ar[r] & \Lambda\oplus\zz\ar[r]^-{\tilde\imath}\ar[d] 
  & \ee\oplus\zz\ar[d]\ar[r]^-{\tilde\beta} 
  & G\ar[d]^-\alpha\ar[r] & 0\\
  0\ar[r] & \Lambda\ar[r]^-\imath\ar[d] & \ee\ar[r]^-\beta\ar[d] 
  & \ee/\Lambda\ar[r]\ar[d] & 0\\
  & 0 & 0 & 0
 }
\]
Indeed, given such an embedding $\tilde\imath$, its
cokernel $\tilde\beta$ gives an extension of $\ee/\Lambda$
by $\zz/m\zz$ represented by the third column 
of the above diagram. Observe that the cokernel
$\tilde\beta$ is unique since $\tilde\imath$ is a 
full-rank homomorphism of free abelian groups.
Moreover equivalent embeddings, that is 
embeddings which differ up to isometries
of the domain $\Lambda\oplus\zz$ and
of the codomain $\ee\oplus\zz$, give
isomorphic extensions of $\ee/\Lambda$
by $\zz/m\zz$.

On the other hand, given an extension of 
$\ee/\Lambda$ by $\zz/m\zz$,
since $\ee\oplus\zz$ is a free abelian group, one 
one can lift the map $\beta$ to a map $\tilde\beta$
whose restriction to $\zz$ surjects onto $\zz/m\zz$.
The kernel $\tilde\imath$ of $\tilde\beta$ gives an
embedding of $\Lambda\oplus\zz$ into $\ee\oplus\zz$.
\end{proof}

\begin{proposition}
Let $\pi\colon X\to\pp^1$ be a minimal elliptic
fibration of Halphen index $2$ on a rational surface.
Then the number of $(-1)$-curves of $X$ is given
in Table~\eqref{halphen}.

{\footnotesize
\begin{table}[h]
\label{halphen}
\begin{tabular}{l|c}
Type & Number of $(-1)$-curves\\
\midrule
$E_8$ & $3$\\
$D_8$ & $6,9$\\
$E_7+A_1$ & $8,10$\\
$A_8$ & $15$\\
$E_6+A_2$ & $18$\\
$A_7+A_1$ & $24,30$\\
$D_5+A_3$ & $28,32$\\
$2A_4$ & $45$\\
$A_5+A_2+A_1$ & $60,66$\\
$D_6+2A_1$ & $26,28$\\
$2D_4$ & $28$\\
$2A_3+2A_1$ & $108,112$\\
$4A_2$ & $144$\\[5pt]
\end{tabular}
\caption{Number of $(-1)$-curves on Halphen elliptic surfaces of index $2$.}
\end{table}
}

\end{proposition}
\begin{proof}
Given a maximal rank sublattice $\Lambda$ of $\ee$,
a positive integer $m$ and an element 
$\xi\in {\rm Ext}^1_{\zz}(\ee/\Lambda,\zz/m\zz)$,
in order to construct a Halphen
fibration of index $m$ whose configuration of 
reducible fibers is of type $\tilde\Lambda$ 
we choose a characteristic
sequence $[p_1,\dots,p_9]$ of points on a smooth
cubic curve $C$ which satisfies the following conditions.
If $\xi$ is the following abelian extension
\[
 \xymatrix@1{
  0\ar[r] & \zz/m\zz\ar[r] & G\ar[r]^-\alpha & \ee/\Lambda\ar[r] & 0
 }
\]
we realize $G$ as a subgroup of the elliptic curve $(C,p_1)$ 
generated by nine points $p_1,\dots,p_9\in C$
such that the element $h=p_1+\dots+p_9$ of order $m$
generates the kernel of $\alpha$ and 
$[p_1,\dots,p_9]\pmod{\langle h \rangle}\equiv
[\bar p_1,\dots,\bar p_9]$, where the second sequence
is the characteristic sequence
in Table~\ref{tab} associated to $\Lambda$.

Given such a characteristic sequence we reconstruct
the set of $(-2)$-curves of the blow-up $X$ of $\pp^2$
at the nine points in the following way.
By Lemma~\ref{effective} each root $R$ of $\Lambda$
lifts to the class of an effective curve of $X$. The set $\mathcal S$
of such lifted classes spans a cone in $\Pic_\qq(X)$.
Fix an ample class $A$ of $\Pic(X)$ and form
the subset $\mathcal S_1$ of $\mathcal S$ consisting 
of elements having minimal intersection with $A$.
Then construct the set $\mathcal S_n$ inductively
as follows. 
Let $\mathcal S_n$ be the subset of 
$\mathcal S \setminus \bigcup_{i \leq n-1}\mathcal S_i$
consisting of elements having minimal 
intersection with $A$ and non-negative
intersection with any element of 
$\bigcup_{i \leq n-1}\mathcal S_i$.
It follows that the set of $(-2)$-curves of $X$ is
\[
 \mathcal C =\bigcup_{i}\mathcal S_i.
\]
According to Lemma~\ref{lifting} the classes
of these curves span a root sublattice of $\kxp
\cong\et$ isometric to $\tilde\Lambda$.
By means of these curves 
we define the lifting $\cl_X
\colon F\to\kxp$ of the map 
$P\colon F\to\kxp/\langle K_X\rangle=\ee$.  
We thus get the class $\delta_X$ of the corresponding 
toric variety, whose Riemann-Roch dimension
is the number of $(-1)$-curves of $X$.

Now assume $m = 2$.
Computing the multigraded Hilbert 
series of the toric varieties given in the appendix
and excluding the cases where the trivial 
extension of $\ee/\Lambda$ by $\zz/2\zz$
needs more than two generators, one directly checks
that all the remaining cases are realized.
\end{proof}

\section*{Appendix}
We provide here the matrix $Q$ defining
the linear map $E\to\Cl(Z)$ for each type of
lattice analysed in this note.

{\footnotesize
\begin{longtable}{lll}
\label{gradings}
Type & Class group & Grading matrix\\
\\
\hline
\\
$E_8$
&
$\zz$
&
$\begin{bmatrix} 1&2&3&4&5&6&4&3&2 \end{bmatrix}$\\
\\
$D_8$ 
&
$\zz\oplus\zz/2\zz$
&
$
\begin{bmatrix}
 1&1&1&1&2&2&2&2&2\\
 \bar 0&\bar 0&\bar 1&\bar 1&\bar 0&\bar 0&\bar 0&\bar 1&\bar 1
\end{bmatrix}
$\\
\\
$A_8$
&
$\zz\oplus\zz/3\zz$
&
$
\begin{bmatrix}
 1&1&1&1&1&1&1&1&1\\
 \bar 0&\bar 1&\bar 2&\bar 0&\bar 1&\bar 2&\bar 0&\bar 1&\bar 2
\end{bmatrix}
$\\
\\
$E_7+A_1$
&
$\zz^2\oplus\zz/2\zz$
&
$
\begin{bmatrix}
1&2&3&4&3&2&1&2&0&0\\
0&0&0&0&0&0&0&0&1&1\\
\bar 0&\bar 0 &\bar 0&\bar 0 &\bar 1 &\bar 1 &\bar 1 &\bar 0 &\bar 0 &\bar 1 
\end{bmatrix}
$\\
\\
$E_6+A_2$
&
$\zz^2\oplus\zz/3\zz$
&
$
\begin{bmatrix}
1&2&3&2&1&2&1&0&0&0\\
0&0&0&0&0&0&0&1&1&1\\
\bar 0& \bar 0& \bar 0& \bar 1& \bar 1& \bar 2& \bar 1& \bar 0& \bar 1& \bar 2
\end{bmatrix}
$\\
\\
$A_7+A_1$
&
$\zz^2\oplus\zz/4\zz$
&
$
\begin{bmatrix}
1&1&1&1&1&1&1&1&0&0\\
0&0&0&0&0&0&0&0&1&1\\
\bar 0& \bar 1& \bar 2& \bar 3& \bar 0& \bar 1& \bar 2& \bar 3& \bar 0& \bar 2
\end{bmatrix}
$\\
\\
$D_5+A_3$
&
$\zz^2\oplus\zz/4\zz$
&
$
\begin{bmatrix}
1&1&1&1&2&2&0&0&0&0\\
0&0&0&0&0&0&1&1&1&1\\
\bar 0& \bar 1& \bar 2& \bar 3& \bar 0& \bar 2& \bar 0& \bar 1& \bar 2& \bar 3
\end{bmatrix}
$\\
\\
$2A_4$
&
$\zz^2\oplus\zz/5\zz$
&
$
\begin{bmatrix}
1&1&1&1&1&0&0&0&0&0\\
0&0&0&0&0&1&1&1&1&1\\
\bar 0& \bar 1& \bar 2& \bar 3& \bar 4& \bar 0& \bar 1& \bar 2& \bar 3& \bar 4
\end{bmatrix}
$\\
\\
$2D_4$
&
$\zz^2\oplus(\zz/2\zz)^2$
&
$
\begin{bmatrix}
1&1&1&1&2&0&0&0&0&0\\
0&0&0&0&0&1&1&1&1&2\\
\bar 0& \bar 1& \bar 0& \bar 1& \bar 0& \bar 0& \bar 1& \bar 0& \bar 1& \bar 0\\
\bar 0& \bar 0& \bar 1& \bar 1& \bar 0& \bar 0& \bar 0& \bar 1& \bar 1& \bar 0
\end{bmatrix}
$\\
\\
$A_5+A_2+A_1$
&
$\zz^3\oplus\zz/6\zz$
&
$
\begin{bmatrix}
1&1&1&1&1&1&0&0&0&0&0\\
0&0&0&0&0&0&1&1&1&0&0\\
0&0&0&0&0&0&0&0&0&1&1\\
\bar 0& \bar 1& \bar 2& \bar 3& \bar 4& \bar 5& \bar 0& \bar 2& \bar 4& \bar 0& \bar 3
\end{bmatrix}
$\\
\\
$D_6+2A_1$
&
$\zz^3\oplus(\zz/2\zz)^2$
&
$
\begin{bmatrix}
1&1&1&1&2&2&2&0&0&0&0\\
0&0&0&0&0&0&0&1&1&0&0\\
0&0&0&0&0&0&0&0&0&1&1\\
\bar 0& \bar 1& \bar 0& \bar 1& \bar 0& \bar 0& \bar 1& \bar 0& \bar 1& \bar 0& \bar 0\\
\bar 0& \bar 0& \bar 1& \bar 1& \bar 0& \bar 0& \bar 0& \bar 0& \bar 1& \bar 0& \bar 1
\end{bmatrix}
$\\
\\
$2A_3+2A_1$
&
$\zz^4\oplus\zz/2\zz\oplus\zz/4\zz$
&
$
\begin{bmatrix}
1&1&1&1&0&0&0&0&0&0&0&0\\
0&0&0&0&1&1&1&1&0&0&0&0\\
0&0&0&0&0&0&0&0&1&1&0&0\\
0&0&0&0&0&0&0&0&0&0&1&1\\
\bar 0& \bar 1& \bar 2& \bar 3& \bar 0& \bar 1& \bar 2& \bar 3& \bar 0& \bar 2& \bar 0& \bar 0\\
\bar 0& \bar 1& \bar 0& \bar 1& \bar 0& \bar 0& \bar 0& \bar 0& \bar 0& \bar 1& \bar 0& \bar 1
\end{bmatrix}
$\\
\\
$4A_2$
&
$\zz^4\oplus(\zz/3\zz)^2$
&
$
\begin{bmatrix}
1&1&1&0&0&0&0&0&0&0&0&0\\
0&0&0&1&1&1&0&0&0&0&0&0\\
0&0&0&1&0&0&1&1&1&0&0&0\\
0&0&0&0&0&0&0&0&0&1&1&1\\
\bar 0& \bar 1& \bar 2& \bar 0& \bar 0& \bar 0& \bar 0& \bar 1& \bar 2& \bar 0& \bar 1& \bar 2\\
\bar 0& \bar 0& \bar 0& \bar 0& \bar 1& \bar 2& \bar 0& \bar 1& \bar 2& \bar 0& \bar 2& \bar 1
\end{bmatrix}
$\\[30pt]
\caption{Grading matrices}
\end{longtable}
}

\end{document}